\documentclass[final,3p,times]{elsarticle}

\usepackage{amsmath,amsthm,amscd,amsfonts,amssymb,epic,eepic,bbm,graphicx}
\usepackage{tikz, tkz-euclide}
\usepackage{ytableau}

\newtheorem{thm}{Theorem}[section]
\newtheorem{cor}[thm]{Corollary}

\newtheorem{lem}[thm]{Lemma}
\newtheorem{defn}[thm]{Definition}

\newtheorem{prop}[thm]{Proposition}

\newtheorem{rem}[thm]{Remark}
\newtheorem{example}[thm]{Example}
\theoremstyle{remark}

\newcommand{\la}{\lambda}
\def\multiset#1#2{\ensuremath{\left(\kern-.5em\left(\genfrac{}{}{0pt}{}{#1}{#2}\right)\kern-.5em\right)}}

\journal{}

\begin{document}

\begin{frontmatter}


\title{Self-conjugate $(s,s+d,\dots,s+pd)$-core partitions and \\ free rational Motzkin paths}

\author[1]{Hyunsoo Cho}
\ead{hyunsoo@ewha.ac.kr}
\author[2]{JiSun Huh\corref{cor1}}
\ead{hyunyjia@ajou.ac.kr}
\cortext[cor1]{Corresponding author.}

\address[1]{Institute of Mathematical Sciences, Ewha Womans University, Seoul 03760, Republic of Korea}
\address[2]{Department of Mathematics, Ajou University, Suwon 16499, Republic of Korea}


\begin{abstract}
A partition is called an $(s_1,s_2,\dots,s_p)$-core partition if it is simultaneously an $s_i$-core for all $i=1,2,\dots,p$. Simultaneous core partitions have been actively studied in various directions. In particular, researchers concerned with properties of such partitions when the sequence of $s_i$ is an arithmetic progression. 

In this paper, for $p\geq 2$ and relatively prime positive integers $s$ and $d$, we propose the $(s+d,d;a)$-abacus of a self-conjugate partition and establish a bijection between the set of self-conjugate $(s,s+d,\dots,s+pd)$-core partitions and the set of free rational Motzkin paths with appropriate conditions. For $p=2,3$, we give formulae for the number of self-conjugate $(s,s+d,\dots,s+pd)$-core partitions and the number of self-conjugate $(s,s+1,\dots,s+p)$-core partitions with $m$ corners.
\end{abstract}



\end{frontmatter}



\section{Introduction}

A \emph{partition} $\la$ of $n$ is a finite non-increasing positive integer sequence $(\la_1,\la_2,\dots,\la_\ell)$ such that the sum of all $\la_i$ is equal to $n$.
For a partition $\la$ of $n$, the \emph{Young diagram} of $\la$ is a collection of $n$ boxes arranged in left-justified rows, with the $i$th row having $\la_i$ boxes. The \emph{hook length} of a box in the Young diagram of $\la$ is defined to be the number of boxes weakly below and strictly to the right of the box. Figure~\ref{fig:young} shows the Young diagram of $\la$ and the hook lengths of each box. For a positive integer $s$, a partition $\la$ is an \emph{$s$-core} (partition) if it has no box of hook length $s$. Referring to Figure~\ref{fig:young}, we can easily verify that $\la=(5,4,2,1)$ is an $s$-core for $s=5,7,$ or $s\geq 9$. For a sequence $(s_1,s_2,\dots,s_p)$ of distinct positive integers, we say that a partition $\la$ is an \emph{$(s_1,s_2,\dots,s_p)$-core} if it is simultaneously an $s_1$-core, an $s_2$-core, \dots, and an $s_p$-core. 

\begin{figure}[ht]\label{fig:young}
\centering
\begin{ytableau}
8&6&4&3&1 \\
6&4&2&1 \\
3&1 \\
1
\end{ytableau}
\caption{The Young diagram and the hook lengths of $\la=(5,4,2,1)$.}
\end{figure}

There has been considerable interest in recent years in core partitions, which have applications in representation theory and number theory.
Since the following result of Anderson \cite{Anderson}, many researchers found results on $(s,t)$-core partitions (see \cite{AHJ,CHW,Fayers,Fayers2,FV,FMS}).

\begin{thm}\cite[Theorem 1]{Anderson}
For relatively prime positive integers $s$ and $t$, the number of $(s,t)$-core partitions is given by
\[
\frac{1}{s+t}\binom{s+t}{s}\,.
\]
In particular, the number of $(s,s+1)$-core partitions is the $s$th Catalan number $C_s=\frac{1}{s+1}\binom{2s}{s}=\frac{1}{2s+1}\binom{2s+1}{s}$.
\end{thm}

Going further, researchers considered simultaneous core partitions when the sequence of $s_i$ forms an arithmetic progression (see \cite{AL,BNY,CHS,Wang,YYZ,YZZ}). Recently, the authors \cite{CHS2} enumerated the $(s,s+d,\dots,s+pd)$-core partitions by giving a lattice path interpretation.

A \emph{free rational Motzkin path of type $(s,t)$} is a path from $(0,0)$ to $(s,t)$ consisting of up steps $U=(1,1)$, down steps $D=(1,-1)$, and flat steps $F=(1,0)$. We denote the set of all free rational Motzkin paths of type $(s,t)$ by $\mathcal{F}(s,t)$. In addition, if a free rational Motzkin path of type $(s,t)$ is allowed to stay only weakly above the line $y=x$, then it is called a \emph{rational Motzkin path} of type $(s,t)$.

\begin{thm}\cite[Theorem 1.5]{CHS2}\label{thm:rational}
Let $s$ and $d$ be relatively prime positive integers. For an integer $p\geq2$, the number of $(s,s+d,\dots,s+pd)$-core partitions is equal to the number of rational Motzkin paths of type $(s+d,-d)$ 
without $UF^iU$ steps for $i=0,1,\dots,p-3$ if $p\geq 3$, that is given by
\[
\frac{1}{s+d}\binom{s+d}{d}+\sum_{k=1}^{\lfloor s/2 \rfloor}\sum_{\ell=0}^{r}\frac{1}{k+d}\binom{k+d}{k-\ell}\binom{k-1}{\ell}\binom{s+d-\ell(p-2)-1}{2k+d-1}\,,
\]
where $r=\min(k-1,\,\lfloor (s-2k)/(p-2) \rfloor)$.
\end{thm}
 
For a partition $\la$, the \emph{conjugate} of $\la$ is the partition whose Young diagram is the reflection along the main diagonal of the Young diagram of $\la$. A partition whose conjugate is equal to itself is called \emph{self-conjugate}. Some researchers have considered with self-conjugate simultaneous core partitions whose cores line up with an arithmetic progression. We denote the set of all self-conjugate $(s,s+d,\dots,s+pd)$-core partitions by $\mathcal{SC}_{(s,s+d,\dots,s+pd)}$. Ford-Mai-Sze \cite{FMS} first investigated self-conjugate $(s,t)$-core partitions and obtained the following result. 

\begin{thm}\cite[Theorem 1]{FMS} \label{thm:FMS}
For relatively prime positive integers $s$ and $t$, the number of self-conjugate $(s,t)$-core partitions is given by
\[
\binom{\lfloor \frac{s}{2} \rfloor + \lfloor \frac{t}{2} \rfloor}{\lfloor \frac{s}{2} \rfloor}\,.
\]
In particular, the number of self-conjugate $(s,s+1)$-core partitions is equal to the number of symmetric Dyck paths of order $s$, that is given by $\binom{s}{\lfloor s/2 \rfloor}$.
\end{thm}

Motivated by Theorem \ref{thm:FMS}, the authors \cite{CHS} gave a formula for the number of self-conjugate $(s,s+1,s+2)$-core partitions by showing that it is equal to the number of symmetric Motzkin paths of length $s$.

\begin{thm}\cite[Theorem 4]{CHS}\label{thm:symMotzkin}
For a positive integer $s$, the number of self-conjugate $(s,s+1,s+2)$-cores is
$$\sum_{i\geq0}\binom{\lfloor \frac{s}{2}\rfloor}{i}\binom{i}{\lfloor \frac{i}{2} \rfloor},$$
which counts the number of symmetric Motzkin paths of length $s$.
\end{thm}

The authors also suggested a conjecture about a relation between the set of self-conjugate $(s,s+1,\dots, s+p)$-core partitions and the set of symmetric $(s,p)$-generalized Dyck paths (see \cite{CHS, YYZ} for the definition of symmetric $(s,p)$-generalized Dyck paths). This conjecture was recently proved by Yan-Yu-Zhou \cite{YYZ}.

\begin{thm}\cite[Theorems 2.14, 2.19, and 2.22]{YYZ} \label{thm:YYZ}
For positive integers $s$ and $p$, the number of self-conjugate $(s,s+1,\dots,s+p)$-core partitions is equal to the number of symmetric $(s,p)$-generalized Dyck paths.
\end{thm}

By using the above theorem together with an another path interpretation, the authors \cite{CHS2} found a formula for the number of self-conjugate $(s,s+1,\dots, s+p)$-core partitions.

\begin{thm}\cite[Theorem 3.9]{CHS2}\label{thm:selfone}
For positive integers $s$ and $p\geq 2$, the number of self-conjugate $(s,s+1,\dots,s+p)$-core partitions is given by
\[
1+\sum_{k=1}^{\lfloor s/2 \rfloor}\sum_{\ell=0}^{r}\binom{\lfloor \frac{k-1}{2} \rfloor}{\lfloor \frac{\ell}{2} \rfloor} \binom{\lfloor \frac{k}{2} \rfloor}{\lfloor \frac{\ell+1}{2} \rfloor}\binom{\lfloor \frac{s-\ell(p-2)}{2} \rfloor}{k}\,,
\]
where $r=\min(k-1,\,\lfloor (s-2k)/(p-2) \rfloor)$.
\end{thm}

As there was no known result on self-conjugate $(s,s+d,\dots,s+pd)$-core partitions with $d\geq 1$ so far, we investigate some properties about them. This paper is organized as follows. In Section~\ref{sec:2}, we first introduce the ``$(s+d,d;a)$-abacus diagram" and define the ``$(s+d,d;a)$-abacus function" of a self-conjugate partition $\la$ and then investigate properties of the $(s+d,d;a)$-abacus function of a self-conjugate $(s,s+d,\dots, s+pd)$-core partition. After that, in Section~\ref{sec:3.1}, we give a lattice path interpretation for a self-conjugate $(s,s+d,\dots,s+pd)$-core partition by constructing injections from the set of self-conjugate $(s,s+d,\dots,s+pd)$-core partitions to the set of free rational Motzkin paths. The following theorem is the main result of this paper.

\begin{thm}\label{thm:main}
Let $s$ and $d$ be relatively prime positive integers. For a given integer $p\geq 2$, the mapping $\phi_{(s+d,d)}$ gives a one-to-one correspondence between $\mathcal{SC}_{(s,s+d,\dots,s+pd)}$ and the set of paths $P\in\mathcal{F}(\lfloor s/2 \rfloor+\lceil d/2 \rceil, -\lceil d/2 \rceil)$ satisfying that i) $P$ has no $UF^iU$ as a consecutive subpath for all $i=0,1,\dots,p-3$ if $p\geq 3$; ii) $P$ starts with no $F^jU$; iii) $P$ ends with no $UF^k$ for all   
\begin{enumerate}
\item [(a)] $j=0,1,\dots,\lfloor (p-4)/2 \rfloor$ if $p\geq 4$ and $k=0,1,\dots, \lfloor (p-3)/2 \rfloor$ if $p\geq 3$, when $s$ is odd and $d$ is even;
\item [(b)] $j=0,1,\dots,\lfloor (p-4)/2 \rfloor$ if $p\geq 4$ and $k=0,1,\dots, p-2$, when $s$ is odd and $d$ is odd;
\item[(c)] $j=0,1, \dots, \lfloor (p-3)/2 \rfloor$ if $p\geq 3$ and $k=0,1, \dots, p-2$, when $s$ is even and $d$ is odd.
\end{enumerate}
\end{thm}

As a corollary, we give formulae for the number of self-conjugate $(s,s+d,s+2d)$-core partitions and the number of self-conjugate $(s,s+d,s+2d,s+3d)$-core partitions in Section~\ref{sec:3.2}. 
In Section \ref{sec:3.3}, we focus on self-conjugate $(s,s+1,\dots s+p)$-core partitions with $m$ corners. We specify free rational Motzkin paths corresponding to self-conjugate $(s,s+1,\dots,s+p)$-core partitions with $m$ corners. In particular, we count the number of self-conjugate $(s,s+1,s+2)$-core partitions with $m$ corners.


\section{The $(s+d,d;a)$-abacus diagram}\label{sec:2}

Let $\la$ be a partition. The $s$-abacus of $\la$, introduced by James-Kerber \cite{JK}, has played important roles in the theory of core partitions (see \cite{AHJ, Fayers, Fayers2, Johnson2, NS}).  
In \cite{CHS2}, the authors introduced the $(s+d,d)$-abacus of $\la$, which is useful for determining whether $\la$ is an $(s,s+d, \dots, s+pd)$-core.
Now, we slightly modify the $(s+d,d)$-abacus of $\la$ to get an abacus, which is useful when we deal with \emph{self-conjugate} $(s,s+d, \dots, s+pd)$-core partitions.     
Note that if $\la$ is a self-conjugate partition, then elements in $MD(\la)$ are all distinct and odd, where 
$MD(\la)$ denotes the set of the main diagonal hook lengths of $\la$. Let $\mathbb{Z}$ denote the set of all integers.

\begin{defn}  
Let $s$ and $d$ be relatively prime positive integers. 
We define the \emph{$(s+d,d;a)$-abacus diagram} to be
the bottom and left justified diagram with infinitely many rows labeled by $i\in \mathbb{Z}$ and $\lfloor (s+d+1)/2 \rfloor$ columns labeled by $j\in\{0,1,\dots,\lfloor (s+d-1)/2 \rfloor\}$ whose position $(i,j)$ is labeled by $a+2(s+d)i+2dj$, 
where $a=a(s,d)$ denotes the integer such that $-a$ is the smallest odd number among $s$ and $s+d$.

For a self-conjugate partition $\la$, the \emph{$(s+d,d;a)$-abacus of $\la$} is obtained from the $(s+d,d;a)$-abacus diagram by placing a \emph{bead} on each position labeled by $\ell$, where $|\ell|\in MD(\la)$. A position without bead is called a \emph{spacer}.
\end{defn}

The following proposition shows that, for self-conjugate partitions $\la$, the $(s+d,d;a)$-abacus of $\la$ are well-defined and all distinct.
In addition, if $\la$ is an $(s,s+d,\dots,s+pd)$-core, then there are exactly $|MD(\la)|$ beads on
the $(s+d,d;a)$-abacus of $\la$.

\begin{prop} \label{prop:injection}
Let $s$ and $d$ be relatively prime positive integers. For every positive odd integer $h$ such that $h\not\equiv s+d \pmod{2(s+d)}$, there exist a unique position labeled by $h$ or $-h$ in the $(s+d,d;a)$-abacus diagram.
\end{prop}

\begin{proof}
For the $(s+d,d;a)$-abacus diagram, positions in column $j$ are labeled by $a+2(s+d)i+2dj$ for $i\in \mathbb{Z}$. Note that the absolute value of these labels are congruent to $a+2dj$ or $2(s+d)-a-2dj$ modulo $2(s+d)$. We claim that $a+2dj$ and $2(s+d)-a-2dj$ for  $j\in\{0,1,\dots, \lfloor (s+d-1)/2 \rfloor$\} are all incongruent modulo $2(s+d)$ except $a+2dj\equiv s+d \pmod{2(s+d)}$. 
First, for $0 \leq j_1 < j_2\leq \lfloor (s+d-1)/2 \rfloor$, it is clear that $a+2dj_1$ and $a+2dj_2$ are incongruent modulo $2(s+d)$.
Now, suppose that $a+2dj_1 \equiv 2(s+d)-a-2dj_2 \pmod{2(s+d)$} for some $0 \leq j_1,j_2\leq \lfloor (s+d-1)/2 \rfloor$. It follows that $2a+2d(j_1+j_2)$ is a multiple of $2(s+d)$. We consider two cases according to the parity of $s$:

If $s$ is odd, $a$ is supposed to be $-s$ so that $-2s+2d(j_1+j_2)$ is a multiple of $2(s+d)$. It follows that $d(j_1+j_2+1)$ is a multiple of $s+d$. Since $s$ and $d$ are relatively prime, the only possibility is that $j_1=j_2=(s+d-1)/2$ when $d$ is even. In this case, $-s+2d(s+d-1)/2 \equiv s+d \pmod{2(s+d)}$.

If $s$ is even, $a$ is supposed to be $-s-d$ so that $-2s-2d+2d(j_1+j_2)$ is a multiple of $2(s+d)$. It follows that $d(j_1+j_2)$ is a multiple of $s+d$. Since $s$ and $d$ are relatively prime, the only possibility is that $j_1=j_2=0$. Note that $-s-d \equiv s+d \pmod{2(s+d)}$.

This completes the proof of the claim. From the claim we have that, for every odd integer $h$, there exists $j\in\{0,1,\dots, \lfloor (s+d-1)/2 \rfloor\}$ such that $h$ is congruent to $a+2dj$ or $2(s+d)-a-2dj$ modulo $2(s+d)$. In addition, if $h\not\equiv s+d \pmod{2(s+d)}$, then there exists a unique position labeled by $h$ or $-h$ in the $(s+d,d;a)$-abacus diagram. 
\end{proof}

By using the following proposition, we investigate the properties of the $(s+d,d;a)$-abacus of a self-conjugate $(s,s+d, \dots, s+pd)$-core partition according to the parity of $s$ and $d$. 

\begin{prop}\cite[Proposition 3]{FMS}\label{prop:FMS}
Let $\la$ be a self-conjugate partition. Then $\la$ is an $s$-core if and only if both of the following hold:
\begin{enumerate}
\item[(a)] If $h\in MD(\la)$ with $h>2s$, then $h-2s\in MD(\la)$.
\item[(b)] If $h_1,h_2\in MD(\la)$, then $h_1+h_2\not\equiv 0 \pmod{2s}$. 
\end{enumerate}
\end{prop}

\begin{cor}\label{cor:sum}
If $\la$ is a self-conjugate $(s,t)$-core partition, then $s+t \notin MD(\la)$.
\end{cor}

\begin{proof} We may assume that $s<t$. Suppose that $s+t \in MD(\la)$, it follows from Proposition~\ref{prop:FMS}~(a) that $t-s \in MD(\la)$. But this gives a contradiction to Proposition~\ref{prop:FMS}~(b) since $(s+t)+(t-s)=2t$. Thus, $s+t\notin MD(\la)$.
\end{proof}

For the $(s+d,d;a)$-abacus of a self-conjugate $(s,s+d,\dots,s+pd)$-core partition $\la$ with $p\geq 2$, let $r(j)$ denote the row number such that position $(r(j),j)$ is labeled by a positive integer while position $(r(j)-1,j)$ is labeled by a negative integer. 

\begin{lem}\label{lem:beads}
Let $\la$ be a self-conjugate partition. For relatively prime positive integers $s$ and $d$, if $\la$ is an  $(s,s+d,\dots,s+pd)$-core, then the $(s+d,d;a)$-abacus of $\la$ satisfies the following.

\begin{enumerate}
\item[(a)] If a bead is placed on position $(i,j)$ such that $i> r(j)$, then a bead is also placed on each of positions $(i-1,j), (i-2,j), \dots, (r(j),j)$.
\item[(b)] If a bead is placed on position $(i,j)$ such that $i< r(j)-1$, then a bead is also placed on each of positions $(i+1,j), (i+2,j), \dots, (r(j)-1,j)$. 
\item[(c)] A bead can be placed on at most one of the two positions $(r(j),j)$ and $(r(j)-1,j)$.
\end{enumerate}
\end{lem}

\begin{proof}
Fix a column number $j$.
\begin{enumerate} 
\item[(a)] A bead is placed on position $(i,j)$ with $i>r(j)$ means that $a+2(s+d)i+2dj \in MD(\la)$. Since $\la$ is an $(s+d)$-core, $a+2(s+d)(i-1)+2dj$ belongs to $MD(\la)$ by Proposition~\ref{prop:FMS}~(a). In a similar way, we also have $a+2(s+d)(i-2)+2dj, \dots, a+2(s+d)r(j)+2dj \in MD(\la)$. Hence, a bead is placed on each of positions $(i-1,j), (i-2,j), \dots, (r(j),j)$.
\item[(b)] A bead is placed on position $(i,j)$ such that $i<r(j)-1$ means $-a-2(s+d)i-2dj \in MD(\la)$. Again, it follows from Proposition~\ref{prop:FMS}~(a) that $-a-2(s+d)(i+1)-2dj$ belongs to $MD(\la)$. In a similar way, we also have $-a-2(s+d)(i+2)-2dj, \dots, -a-2(s+d)(r(j)-1)-2dj \in MD(\la)$. Hence, a bead is placed on each of positions $(i+1,j), (i+2,j), \dots, (r(j)-1,j)$.
\item[(c)] By Proposition~\ref{prop:FMS}~(b), at most one of $a+2(s+d)r(j)+2dj$ and $-a-2(s+d)(r(j)-1)-2dj$ belongs to $MD(\la)$ since the sum of those two numbers is $2(s+d)$. Therefore, a bead can be placed on at most one of the positions $(r(j),j)$ and $(r(j)-1,j)$. 
\end{enumerate}
\end{proof}

For $p\geq 2$, let $\la$ be a self-conjugate $(s,s+d,\dots, s+pd)$-core partition.
In order to explain the properties of the $(s+d,d;a)$-abacus of $\la$ more simply, we define the \emph{$(s+d,d;a)$-abacus function of $\la$}
\[
f:\{0,1,\dots,\lfloor (s+d-1)/2 \rfloor\}\rightarrow \mathbb{Z}
\]
as follows:
If a bead is placed on a position in column $j$ being labeled by a positive integer, then $f(j)$ is defined to be the largest number $i$ such that a bead is placed on position $(i,j)$; Otherwise, $f(j)$ is defined to be the largest number $i$ such that position $(i,j)$ is a spacer being labeled by a negative integer. 

The following proposition gives some basic properties of the $(s+d,d;a)$-abacus function of a self-conjugate $(s,s+d,\dots,s+pd)$-core partition. 

\begin{prop}\label{prop:f_initial}
Let $s$ and $d$ be relatively prime positive integers. For $p\geq 2$, if $\la$ is a self-conjugate $(s,s+d,\dots,s+pd)$-core partition, then the $(s+d,d;a)$-abacus function $f$ of $\la$ satisfies the following.
\begin{enumerate}
\item[(a)] $f(0)=0$.
\item[(b)] $f(j-1)$ is equal to one of the three values $f(j)-1$, $f(j)$, and $f(j)+1$, for $j=1,\dots, (s+d-1)/2$.
\item[(c)] If $p\geq 3$ and $f(j-1)=f(j)-1$, then $f(j-p+1),f(j-p+2),\dots,f(j-2)\geq f(j-1)$, for $j=p-1, \dots, (s+d-1)/2$.
\end{enumerate}
 \end{prop}

\begin{proof}
We consider the $(s+d,d;a)$-abacus of $\la$. 
\begin{enumerate}
\item[(a)] Since position $(0,0)$ is labeled by $a=-s$ (resp. $a=-s-d$) when $s$ is odd (resp. even) and position $(1,0)$ is labeled by $s+2d$ (resp. $s+d$), both of positions are spacers and $r(0)=1$. It follows from Lemma~\ref{lem:beads} that there is no bead in column $0$. Hence, $f(0)=0$.
\item[(b)] For a fixed $j$, let $f(j)=i$. 

Suppose that a bead is placed on position $(i,j)$ so that it is labeled by a positive integer. We first show that $f(j-1)\geq f(j)-1$ by considering position $(i-1,j-1)$. If position $(i-1,j-1)$ is labeled by a positive integer, then a bead is placed on position $(i-1,j-1)$ by Proposition~\ref{prop:FMS}~(a) as $\la$ is an $(s+2d)$-core. Otherwise, if position $(i-1,j-1)$ is labeled by a negative integer, then position $(i-1,j-1)$ is a spacer by Proposition~\ref{prop:FMS}~(b). In any case, it follows from the definition of $f$ that $f(j-1)\geq f(j)-1$. Now, we show that $f(j-1)\leq f(j)+1$ by considering position $(i+2,j-1)$. Since $f(j)=i$ and a bead is placed on position $(i,j)$, position $(i+1,j)$ is a spacer. Therefore, position $(i+2,j-1)$ is a spacer by Proposition~\ref{prop:FMS}~(a) as $\la$ is an $s$-core. Hence, $f(j-1)\leq f(j)+1$.

Now, suppose that position $(i,j)$ is a spacer so that it is labeled by a negative integer. Since position $(i-1,j-1)$ is labeled by a negative integer, it is a spacer as $\la$ is an $(s+2d)$-core. Therefore, $f(j-1)\geq f(j)-1$. To complete the proof,
we assume that $f(j-1)\geq i+2$. If position $(i+2,j-1)$ is labeled by a positive integer, then a bead is placed on this position by Lemma~\ref{lem:beads}~(a). In this case, either a bead is placed on position $(i+1,j)$ being labeled by a positive integer or position $(i+1,j)$ is a spacer being labeled by a negative integer by Proposition~\ref{prop:FMS}~(a) and (b) as $\la$ is an $s$-core. It contradicts to $f(j)=i$. Otherwise, if position $(i+2,j-1)$ is labeled by a negative integer, then it is a spacer. Therefore, position $(i+1,j)$ is a spacer by Proposition~\ref{prop:FMS}~(a). Also, it contradicts to $f(j)=i$. Hence, $f(j-1)\leq f(j)+1$.

\item[(c)] For a fixed $j$, let $f(j)=i$. It suffices to show that $f(j-k)\geq f(j)-1$, for $k=2,\dots,p-1$. Note that $\la$ is an $(s+(k+1)d)$-core partition. 
If position $(i-1,j-k)$ is labeled by a negative integer, then it is a spacer regardless of the presence of a bead on position $(i,j)$ by Proposition~\ref{prop:FMS}~(a) and (b).
If position $(i-1,j-k)$ is labeled by a positive integer, then a bead is placed on position $(i,j)$ being labeled by a positive integer. By Proposition~\ref{prop:FMS}~(a), a bead is placed on position $(i-1,j-k)$.
In any case, we conclude that $f(j-k)\geq f(j)-1$ by the definition of the $(s+d,d;a)$-abacus function $f$.
\end{enumerate}
\end{proof}

In the following three subsections, 
we cover other properties which is depending on the parity of $s$ and $d$ and give several examples of the $(s+d,d;a)$-abacus diagram and the $(s+d,d;a)$-abacus function of a self-conjugate $(s,s+d,\dots,s+pd)$-core partition. Note that we do not need to consider the case when both of $s$ and $d$ are even because $s$ and $d$ are supposed to be relatively prime positive integers.


\subsection{Self-conjugate (s,s+d,\dots,s+pd)-cores with odd $s$ and even $d$}  

First, we consider the case where $s$ is odd and $d$ is even. In the following proposition, we give additional properties of the $(s+d,d;a)$-abacus function $f$ of a self-conjugate $(s,s+d,\dots,s+pd)$-core partition $\la$, where $s$ is odd and $d$ is even. Recall that $a=a(s,d)=-s$ in this case.

\begin{prop}\label{prop:f_oe}
Let $s$ and $d$ be relatively prime positive integers such that $s$ is odd and $d$ is even.
For $p\geq 2$, if $\la$ is a self-conjugate $(s,s+d,\dots,s+pd)$-core partition, then the $(s+d,d;-s)$-abacus function $f$ of $\la$ satisfies the following.
\begin{enumerate}
\item[(a)] $f((s+d-1)/2 )=-d/2$.
\item[(b)] If $p\geq 3$, then $f((s+d-1)/2-k-1)\geq -d/2$, for $k=0,1,\dots, \lfloor (p-3)/2 \rfloor$.
\item[(c)] If $p\geq 4$, then $f(\ell+1)\leq 0$, for $\ell=0,1,\dots,\lfloor (p-4)/2 \rfloor$.
\end{enumerate}

\end{prop}

\begin{proof}
\begin{enumerate}
\item[(a)] Since position $(-d/2,(s+d-1)/2)$ is labeled by $-s-d$ and position $(-d/2+1,(s+d-1)/2)$ is labeled by $s+d$, both of them are spacers. By Lemma~\ref{lem:beads}~(a) and (b), there is no bead in column $(s+d-1)/2$. Hence, $f((s+d-1)/2)=-d/2$ by the construction of the $(s+d,d;-s)$-abacus function $f$ of $\la$.

\item[(b)] Note that position $(-d/2,(s+d-1)/2-k-1)$ is a spacer because it is labeled by $-s-(2k+3)d$ and $\la$ is an $(s+(2k+3)d)$-core, where $k=0,1,\dots, \lfloor (p-3)/2 \rfloor$. Hence, $f((s+d-1)/2-k-1)\geq -d/2$.
\item[(c)] Note that position $(1,\ell+1)$ is a spacer because it is labeled by $s+(2\ell +4)d$ and $\la$ is an $(s+(2\ell +4)d)$-core, where $\ell =0,1,\dots,\lfloor (p-4)/2 \rfloor$. Hence, $f(\ell+1)\leq 0$.
\end{enumerate}
\end{proof}

\begin{example}\label{ex:abacus_oe} 
Let $\la$ be the self-conjugate partition with $MD(\la)=\{77,41,35,27,19,11,5,3\}$. It follows from Proposition~\ref{prop:FMS} that $\la$ can be considered as a $(21,25,29,33,37)$-core partition.
Figure~\ref{fig:abacus_oe} shows the $(25,4;-21)$-abacus of $\la$ and the path obtained by connecting each pair of the two points $(j-1,f(j-1))$ and $(j,f(j))$ with a straight line segment for $j=1,\dots,12$.
Note that the $(25,4;-21)$-abacus function $f$ of $\la$ is given by
\[
f(0)=f(1)=0, ~f(2)=-1, ~f(3)=f(4)=f(5)=0, ~f(6)=1,
\]
\[
f(7)=0, ~f(8)=-1, ~f(9)=-2, ~f(10)=-3, ~f(11)=f(12)=-2.
\] 
Indeed, one can see that the $(25,4,-21)$-abacus function $f$ of $\la$ agrees with all the properties given in Lemma~\ref{lem:beads} and Propositions~\ref{prop:f_initial} and \ref{prop:f_oe}. 
\end{example}

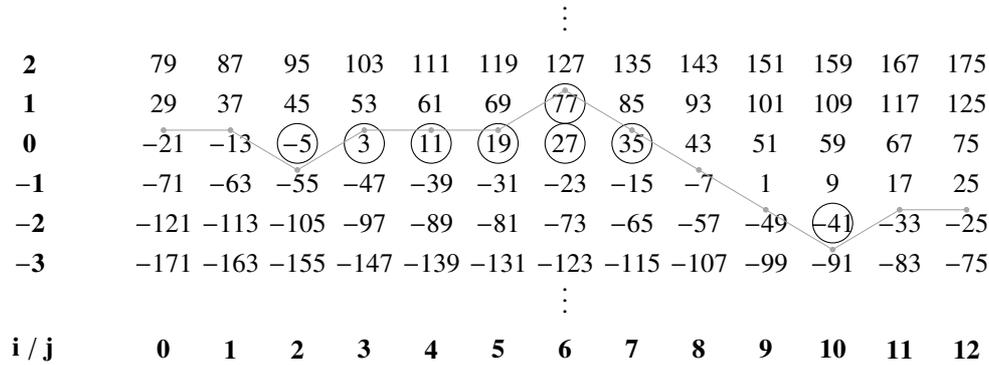
\begin{figure}[ht!]
\centering
\begin{tikzpicture}[scale=.44]

\node at (-4,2.4) {$\mathbf{2}$};
\node at (-4,1.2) {$\mathbf{1}$};
\node at (-4,0) {$\mathbf{0}$};
\node at (-4,-1.2) {$\mathbf{-1}$};
\node at (-4,-2.4) {$\mathbf{-2}$};
\node at (-4,-3.6) {$\mathbf{-3}$};

\node at (-3.9,-6.2) {$\mathbf{i~/~j}$};

\node at (0,-6.2) {$\mathbf{0}$};
\node at (2,-6.2) {$\mathbf{1}$};
\node at (4,-6.2) {$\mathbf{2}$};
\node at (6,-6.2) {$\mathbf{3}$};
\node at (8,-6.2) {$\mathbf{4}$};
\node at (10,-6.2) {$\mathbf{5}$};
\node at (12,-6.2) {$\mathbf{6}$};
\node at (14,-6.2) {$\mathbf{7}$};
\node at (16,-6.2) {$\mathbf{8}$};
\node at (18,-6.2) {$\mathbf{9}$};
\node at (20,-6.2) {$\mathbf{10}$};
\node at (22,-6.2) {$\mathbf{11}$};
\node at (24,-6.2) {$\mathbf{12}$};

\foreach \i in {-21,-13,-5,3,11,19,27,35,43,51,59,67,75}
\node at (\i/4+21/4,0) {$\i$};
\foreach \i in {29,37,45,53,61,69,77,85,93,101,109,117,125}
\node at (\i/4-29/4,1.2) {$\i$};
\foreach \i in {79,87,95,103,111,119,127,135,143,151,159,167,175}
\node at (\i/4-79/4,2.4) {$\i$};
\foreach \i in {-71,-63,-55,-47,-39,-31,-23,-15,-7,1,9,17,25}
\node at (\i/4+71/4,-1.2) {$\i$};
\foreach \i in {-121,-113,-105,-97,-89,-81,-73,-65,-57,-49,-41,-33,-25}
\node at (\i/4+121/4,-2.4) {$\i$};
\foreach \i in {-171,-163,-155,-147,-139,-131,-123,-115,-107,-99,-91,-83,-75}
\node at (\i/4+171/4,-3.6) {$\i$};

\node at (12,4) {\vdots};
\node at (12,-4.5) {\vdots};

\draw (4,0) circle (17pt);
\draw (6,0) circle (17pt);
\draw (8,0) circle (17pt);
\draw (10,0) circle (17pt);
\draw (12,0) circle (17pt);
\draw (14,0) circle (17pt);
\draw (12,1.2) circle (17pt);
\draw (20,-2.4) circle (17pt);

\draw[color=gray!70] (0,0.4)--(2,0.4)--(4,-0.8)--(6,0.4)--(8,0.4)--(10,0.4)--(12,1.6)--(14,0.4)--(16,-0.8)--(18,-2)--(20,-3.2)--(22,-2)--(24,-2);

\filldraw[color=gray!70] 
(0,0.4) circle (2pt)
(2,0.4) circle (2pt)
(4,-0.8) circle (2pt)
(6,0.4) circle (2pt)
(8,0.4) circle (2pt)
(10,0.4) circle (2pt)
(12,1.6) circle (2pt)
(14,0.4) circle (2pt)
(16,-0.8) circle (2pt)
(18,-2) circle (2pt)
(20,-3.2) circle (2pt)
(22,-2) circle (2pt)
(24,-2) circle (2pt)
;
\end{tikzpicture}
\caption{The $(25,4;-21)$-abacus of the self-conjugate partition $\la$ with $MD(\la)=\{77,41,35,27,19,11,5,3\}$}\label{fig:abacus_oe}
\end{figure}

\begin{rem} \label{rem:oe}
For relatively prime positive integers $s$ and $d$ such that $s$ is odd and $d$ is even, let $f$ be a function satisfying all the conditions in Propositions~\ref{prop:f_initial} and \ref{prop:f_oe}. We remark that there exist a unique self-conjugate $(s,s+d,\dots,s+pd)$-core partition $\la$ such that $f$ is the $(s+d,d;-s)$-abacus function of $\la$ by Lemma~\ref{lem:beads} and the construction of $f$. From this fact together with Proposition~\ref{prop:injection}, we can conclude that there is a one-to-one correspondence between the set of self-conjugate $(s,s+d,\dots, s+pd)$-core partitions and the set of functions satisfying all the conditions in Propositions~\ref{prop:f_initial} and \ref{prop:f_oe}.
\end{rem}


\subsection{Self-conjugate $(s,s+d,\dots,s+pd)$-cores with odd $s$ and odd $d$}

Now, we consider the case where both of $s$ and $d$ are odd. Let $\la$ be a self-conjugate partition. 
The following proposition gives several additional properties of the $(s+d,d;a)$-abacus function $f$ of $\la$, where both of $s$ and $d$ are odd. In this case, $a=a(s,d)=-s$ as well.

\begin{prop}\label{prop:f_oo}
Let $s$ and $d$ be relatively prime positive integers such that both of $s$ and $d$ are odd.
For $p\geq 2$, if $\la$ is a self-conjugate $(s,s+d,\dots,s+pd)$-core partition, then the $(s+d,d;-s)$-abacus function $f$ of $\la$ satisfies the following.
\begin{enumerate}
\item[(a)] $f((s+d-2)/2)=-(d-1)/2$ or $-(d+1)/2$.
\item[(b)] If $p \geq 3$, then $f((s+d-2)/2-k-1)\geq -(d+1)/2$, for $k=0,1,\dots, p-3$.
\item[(c)] If $p\geq 4$, then $f(\ell+1)\leq 0$, for $\ell=0,1,\dots,\lfloor (p-4)/2 \rfloor$.
\end{enumerate}

\end{prop}

\begin{proof}
\begin{enumerate}
\item[(a)] Since position $(-(d-1)/2,(s+d-2)/2)$ is labeled by $-d$, position $(-(d-1)/2+1,(s+d-2)/2)$ is labeled by $2s+d$, and position $(-(d-1)/2-1,(s+d-2)/2)$ is labeled by $-2s-3d$. By Corollary \ref{cor:sum}, $2s+d=s+(s+d),2s+3d=(s+d)+(s+2d)\notin MD(\la)$.
It follows from Lemma \ref{lem:beads} that there is at most one bead which is labeled by $-d$ in column  $(s+d-2)/2$. Hence, $f((s+d-2)/2 )=-(d-1)/2$ or $-(d+1)/2$.
\item[(b)] Position $(-(d+1)/2,(s+d-2)/2-k-1)$ is a spacer because it is labeled by $-2s-(2k+5)d=-\{s+(k+2)d\}+\{s+(k+3)d\}$ and $\la$ is an $(s+(k+2)d,s+(k+3)d)$-core, where $k=0,1,\dots, p-3$. Hence, $f((s+d-2)/2-k-1)\geq -(d+1)/2$.
\item[(c)] Position $(1,\ell+1)$ is a spacer because it is labeled by $s+(2\ell+4)d$ and $\la$ is an $(s+(2\ell+4)d)$-core, where $\ell=0,1,\dots,\lfloor (p-4)/2 \rfloor$. Hence, $f(\ell+1)\leq 0$.
\end{enumerate}
\end{proof}

\begin{example}\label{ex:abacus_oo} 
Let $\mu$ be the self-conjugate partition with $MD(\mu)=\{67,65,21,19,15,13,11,9,7,3,1\}$. It follows from Proposition~\ref{prop:FMS} that $\mu$ can be considered as a $(23,26,29,32)$-core partition.
Figure~\ref{fig:abacus_oo} shows the $(26,3;-23)$-abacus of $\mu$ and the path obtained by connecting each pair of the two points $(j-1,f(j-1))$ and $(j,f(j))$ with a straight line segment for $j=1,\dots,12$.
Note that the $(26,3;-23)$-abacus function $f$ of $\mu$ is given by
\[
f(0)=f(1)=0, ~f(2)=-1, ~f(3)=f(4)=f(5)=0, ~f(6)=1,
\]
\[
f(7)=0, ~f(8)=-1, ~f(9)=-2, ~f(10)=-3, ~f(11)=f(12)=-2,
\] 
and the $(26,3,-23)$-abacus function $f$ of $\mu$ agrees with all the properties given in Lemma~\ref{lem:beads} and Propositions~\ref{prop:f_initial} and \ref{prop:f_oo}. 
\end{example}

\begin{figure}[ht!]
\centering
\begin{tikzpicture}[scale=.44]

\node at (-4,2.4) {$\mathbf{2}$};
\node at (-4,1.2) {$\mathbf{1}$};
\node at (-4,0) {$\mathbf{0}$};
\node at (-4,-1.2) {$\mathbf{-1}$};
\node at (-4,-2.4) {$\mathbf{-2}$};
\node at (-4,-3.6) {$\mathbf{-3}$};

\node at (-3.9,-6.2) {$\mathbf{i~/~j}$};

\node at (0,-6.2) {$\mathbf{0}$};
\node at (2,-6.2) {$\mathbf{1}$};
\node at (4,-6.2) {$\mathbf{2}$};
\node at (6,-6.2) {$\mathbf{3}$};
\node at (8,-6.2) {$\mathbf{4}$};
\node at (10,-6.2) {$\mathbf{5}$};
\node at (12,-6.2) {$\mathbf{6}$};
\node at (14,-6.2) {$\mathbf{7}$};
\node at (16,-6.2) {$\mathbf{8}$};
\node at (18,-6.2) {$\mathbf{9}$};
\node at (20,-6.2) {$\mathbf{10}$};
\node at (22,-6.2) {$\mathbf{11}$};
\node at (24,-6.2) {$\mathbf{12}$};

\foreach \i in {-23,-17,-11,-5,1,7,13,19,25,31,37,43,49}
\node at (\i/3+23/3,0) {$\i$};
\foreach \i in {29,35,41,47,53,59,65,71,77,83,89,95,101}
\node at (\i/3-29/3,1.2) {$\i$};
\foreach \i in {71,77,83,89,95,101,107,113,119,125,131,137,143}
\node at (\i/3-71/3,2.4) {$\i$};
\foreach \i in {-75,-69,-63,-57,-51,-45,-39,-33,-27,-21,-15,-9,-3}
\node at (\i/3+75/3,-1.2) {$\i$};
\foreach \i in {-127,-121,-115,-109,-103,-97,-91,-85,-79,-73,-67,-61,-55}
\node at (\i/3+127/3,-2.4) {$\i$};
\foreach \i in {-179,-173,-167,-161,-155,-149,-143,-137,-131,-125,-119,-113,-107}
\node at (\i/3+179/3,-3.6) {$\i$};

\node at (12,4) {\vdots};
\node at (12,-4.5) {\vdots};

\draw (4,0) circle (17pt);
\draw (8,0) circle (17pt);
\draw (10,0) circle (17pt);
\draw (12,0) circle (17pt);
\draw (14,0) circle (17pt);
\draw (12,1.2) circle (17pt);
\draw (18,-1.2) circle (17pt);
\draw (20,-1.2) circle (17pt);
\draw (22,-1.2) circle (17pt);
\draw (24,-1.2) circle (17pt);
\draw (20,-2.4) circle (17pt);

\draw[color=gray!70] (0,0.4)--(2,0.4)--(4,-0.8)--(6,0.4)--(8,0.4)--(10,0.4)--(12,1.6)--(14,0.4)--(16,-0.8)--(18,-2)--(20,-3.2)--(22,-2)--(24,-2);

\filldraw[color=gray!70] 
(0,0.4) circle (2pt)
(2,0.4) circle (2pt)
(4,-0.8) circle (2pt)
(6,0.4) circle (2pt)
(8,0.4) circle (2pt)
(10,0.4) circle (2pt)
(12,1.6) circle (2pt)
(14,0.4) circle (2pt)
(16,-0.8) circle (2pt)
(18,-2) circle (2pt)
(20,-3.2) circle (2pt)
(22,-2) circle (2pt)
(24,-2) circle (2pt)
;
\end{tikzpicture}
\caption{The $(26,3;-23)$-abacus of the self-conjugate partition $\mu$ with $MD(\mu)=\{67,65,21,19,15,13,11,9,7,3,1\}$}\label{fig:abacus_oo}
\end{figure}

\begin{rem} \label{rem:oo}
For relatively prime positive integers $s$ and $d$ such that both $s$ and $d$ are odd, let $f$ be a function satisfying all the conditions in Propositions~\ref{prop:f_initial} and \ref{prop:f_oo}. We remark that there exist a unique self-conjugate $(s,s+d,\dots,s+pd)$-core partition $\la$ such that $f$ is the $(s+d,d;-s)$-abacus function of $\la$ by Lemma~\ref{lem:beads} and the construction of $f$. From this fact together with Proposition~\ref{prop:injection}, we can conclude that there is a one-to-one correspondence between the set of self-conjugate $(s,s+d,\dots, s+pd)$-core partitions and the set of functions satisfying all the conditions in Propositions~\ref{prop:f_initial} and \ref{prop:f_oo}.
\end{rem}


\subsection{Self-conjugate $(s,s+d,\dots,s+pd)$-cores with even $s$ and odd $d$}

Finally, we consider the case where $s$ is even and $d$ is odd. We give additional properties of the $(s+d,d;a)$-abacus function $f$ of a self-conjugate $(s,s+d,\dots,s+pd)$-core partition $\la$, where $s$ is even and $d$ is odd so that $a=-(s+d)$.

\begin{prop}\label{prop:f_eo}
Let $s$ and $d$ be relatively prime positive integers such that $s$ is even and $d$ is odd. For $p\geq 2$, if $\la$ is a self-conjugate $(s,s+d,\dots,s+pd)$-core partition, then the $(s+d,d;-s-d)$-abacus function $f$ of $\la$ satisfies the following.
\begin{enumerate}
\item[(a)] $f((s+d-1)/2)=-(d-1)/2$ or $-(d+1)/2$.
\item[(b)] If $p\geq 3$, then $f((s+d-1)/2-k-1)\geq -(d+1)/2$, for $k=0,1,\dots, p-3$.
\item[(c)] If $p\geq 3$, then $f(\ell+1)\leq 0$, for $\ell=0,1,\dots,\lfloor (p-3)/2 \rfloor$.
\end{enumerate}

\end{prop}

\begin{proof}
\begin{enumerate}
\item[(a)] Since position $(-(d-1)/2,(s+d-1)/2)$ is labeled by $-d$, position $(-(d-1)/2+1,(s+d-1)/2)$ is labeled by $2s+d$, and position $(-(d-1)/2-1,(s+d-1)/2)$ is labeled by $-2s-3d$. By Corollary \ref{cor:sum}, $2s+d=s+(s+d),2s+3d=(s+d)+(s+2d)\notin MD(\la)$.
It follows from Lemma \ref{lem:beads} that there is at most one bead which is labeled by $-d$ in column  $(s+d-1)/2$. Hence, $f((s+d-1)/2 )=-(d-1)/2$ or $-(d+1)/2$.
\item[(b)] Position $(-(d+1)/2,(s+d-1)/2-k-1)$ is a spacer because it is labeled by $-2s-(2k+5)d=-\{s+(k+2)d\}+\{s+(k+3)d\}$ and $\la$ is an $(s+(k+2)d,s+(k+3)d)$-core, where $k=0,1,\dots, p-2$. Hence, $f((s+d-2)/2-k-1)\geq -(d+1)/2$.
\item[(c)] Position $(1,\ell+1)$ is a spacer because it is labeled by $s+(2\ell+3)d$ and $\la$ is an $(s+(2\ell+3)d)$-core, where $\ell=0,1,\dots,\lfloor (p-3)/2 \rfloor$. Hence, $f(\ell+1)\leq 0$.
\end{enumerate}
\end{proof}

\begin{example}\label{ex:abacus_eo} 
Let $\nu$ be the self-conjugate partition with $MD(\nu)=\{65,61,21,17,15,13,11,9,5,3\}$. It follows from Proposition~\ref{prop:FMS} that $\nu$ can be considered as a $(22,25,28,31)$-core partition.
Figure~\ref{fig:abacus_eo} shows the $(25,3;-25)$-abacus of $\nu$ and the path obtained by connecting each pair of the two points $(j-1,f(j-1))$ and $(j,f(j))$ with a straight line segment for $j=1,\dots,12$.
Note that the $(25,3;-25)$-abacus function $f$ of $\nu$ is given by
\[
f(0)=f(1)=0, ~f(2)=-1, ~f(3)=f(4)=f(5)=0, ~f(6)=1,
\]
\[
f(7)=0, ~f(8)=-1, ~f(9)=-2, ~f(10)=-3, ~f(11)=f(12)=-2,
\] 
and the $(25,3,-25)$-abacus function $f$ of $\nu$ agrees with all the properties given in Lemma~\ref{lem:beads} and Propositions~\ref{prop:f_initial} and \ref{prop:f_eo}. 
\end{example}

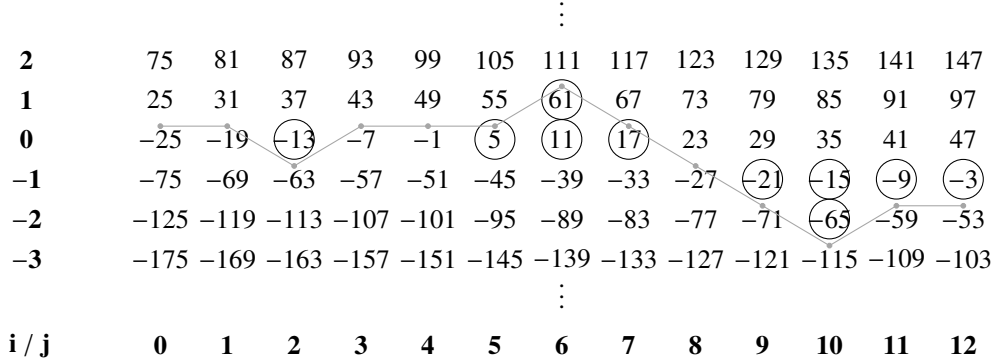
\begin{figure}[ht!]
\centering
\begin{tikzpicture}[scale=.44]

\node at (-4,2.4) {$\mathbf{2}$};
\node at (-4,1.2) {$\mathbf{1}$};
\node at (-4,0) {$\mathbf{0}$};
\node at (-4,-1.2) {$\mathbf{-1}$};
\node at (-4,-2.4) {$\mathbf{-2}$};
\node at (-4,-3.6) {$\mathbf{-3}$};

\node at (-3.9,-6.2) {$\mathbf{i~/~j}$};

\node at (0,-6.2) {$\mathbf{0}$};
\node at (2,-6.2) {$\mathbf{1}$};
\node at (4,-6.2) {$\mathbf{2}$};
\node at (6,-6.2) {$\mathbf{3}$};
\node at (8,-6.2) {$\mathbf{4}$};
\node at (10,-6.2) {$\mathbf{5}$};
\node at (12,-6.2) {$\mathbf{6}$};
\node at (14,-6.2) {$\mathbf{7}$};
\node at (16,-6.2) {$\mathbf{8}$};
\node at (18,-6.2) {$\mathbf{9}$};
\node at (20,-6.2) {$\mathbf{10}$};
\node at (22,-6.2) {$\mathbf{11}$};
\node at (24,-6.2) {$\mathbf{12}$};

\foreach \i in {-25,-19,-13,-7,-1,5,11,17,23,29,35,41,47}
\node at (\i/3+25/3,0) {$\i$};
\foreach \i in {25,31,37,43,49,55,61,67,73,79,85,91,97}
\node at (\i/3-25/3,1.2) {$\i$};
\foreach \i in {75,81,87,93,99,105,111,117,123,129,135,141,147}
\node at (\i/3-75/3,2.4) {$\i$};
\foreach \i in {-75,-69,-63,-57,-51,-45,-39,-33,-27,-21,-15,-9,-3}
\node at (\i/3+75/3,-1.2) {$\i$};
\foreach \i in {-125,-119,-113,-107,-101,-95,-89,-83,-77,-71,-65,-59,-53}
\node at (\i/3+125/3,-2.4) {$\i$};
\foreach \i in {-175,-169,-163,-157,-151,-145,-139,-133,-127,-121,-115,-109,-103}
\node at (\i/3+175/3,-3.6) {$\i$};

\node at (12,4) {\vdots};
\node at (12,-4.5) {\vdots};

\draw (4,0) circle (17pt);
\draw (10,0) circle (17pt);
\draw (12,0) circle (17pt);
\draw (14,0) circle (17pt);
\draw (12,1.2) circle (17pt);
\draw (18,-1.2) circle (17pt);
\draw (20,-1.2) circle (17pt);
\draw (22,-1.2) circle (17pt);
\draw (24,-1.2) circle (17pt);
\draw (20,-2.4) circle (17pt);

\draw[color=gray!70] (0,0.4)--(2,0.4)--(4,-0.8)--(6,0.4)--(8,0.4)--(10,0.4)--(12,1.6)--(14,0.4)--(16,-0.8)--(18,-2)--(20,-3.2)--(22,-2)--(24,-2);

\filldraw[color=gray!70] 
(0,0.4) circle (2pt)
(2,0.4) circle (2pt)
(4,-0.8) circle (2pt)
(6,0.4) circle (2pt)
(8,0.4) circle (2pt)
(10,0.4) circle (2pt)
(12,1.6) circle (2pt)
(14,0.4) circle (2pt)
(16,-0.8) circle (2pt)
(18,-2) circle (2pt)
(20,-3.2) circle (2pt)
(22,-2) circle (2pt)
(24,-2) circle (2pt)
;
\end{tikzpicture}
\caption{The $(25,3;-25)$-abacus of the self-conjugate partition $\nu$ with $MD(\nu)=\{65,61,21,17,15,13,11,9,5,3\}$}\label{fig:abacus_eo}
\end{figure}

\begin{rem} \label{rem:eo}
For relatively prime positive integers $s$ and $d$ such that $s$ is even and $d$ is odd, let $f$ be a function satisfying all the conditions in Propositions~\ref{prop:f_initial} and \ref{prop:f_eo}. We remark that there exist a unique self-conjugate $(s,s+d,\dots,s+pd)$-core partition $\la$ such that $f$ is the $(s+d,d;-s-d)$-abacus function of $\la$ by Lemma~\ref{lem:beads} and the construction of $f$. From this fact together with Proposition~\ref{prop:injection}, we can conclude that there is a one-to-one correspondence between the set of self-conjugate $(s,s+d,\dots, s+pd)$-core partitions and the set of functions satisfying all the conditions in Propositions~\ref{prop:f_initial} and \ref{prop:f_eo}.
\end{rem}

\section{Free rational Motzkin paths of type $(s,t)$ with restrictions}\label{sec:3}

In this section, we give a lattice path interpretation of  self-conjugate $(s,s+d,\dots,s+pd)$-core partitions.

\subsection{Proof of Theorem~\ref{thm:main}}\label{sec:3.1}

For relatively prime positive integers $s$ and $d$, we construct a mapping 
\[
\phi_{(s+d,d)}: \mathcal{SC}_{(s,s+d,s+2d)} \rightarrow \mathcal{F}(\lfloor s/2 \rfloor+\lceil d/2 \rceil, -\lceil d/2 \rceil)
\]
associated with the $(s+d,d;a)$-abacus function $f$ of $\la\in\mathcal{SC}_{(s,s+d,s+2d)}$ as follows: 
First, for convenience, we set $f(\lfloor s/2 \rfloor +(d+1)/2)=-(d+1)/2$ when $d$ is odd. The path $\phi_{(s+d,d)}(\la)$ starts from $(0,0)$ and its $j$th step is $(1,f(j)-f(j-1))$ for all 
$j=1,\dots,\lfloor s/2 \rfloor+\lceil d/2 \rceil$. 

From our first setting and Propositions \ref{prop:f_oe}~(a) and \ref{prop:f_oo}~(a), the last step of the path is either $D$ or $F$ when $d$ is odd.
From this fact together with Proposition~\ref{prop:f_initial}~(b), we can say that the path $\phi_{(s+d,d)}(\la)$ consist of up steps $U=(1,1)$, down steps $D=(1,-1)$, and flat steps $F=(1,1)$.
Since $f(0)=0$ by Proposition~\ref{prop:f_initial}~(a), $f(\lfloor s/2 \rfloor+\lceil d/2 \rceil)=-\lceil d/2 \rceil$ by Proposition~\ref{prop:f_oe}~(a) and by our first setting, we conclude that $\phi_{(s+d,d)}(\la)$ is a free rational Motzkin path of type $(\lfloor s/2 \rfloor+\lceil d/2 \rceil),-\lceil d/2 \rceil)$. Hence, the mapping $\phi_{(s+d,d)}$ is well-defined. We note that it is possible that the two paths $\phi_{(s_1+d_1,d_1)}(\mu)$ and $\phi_{(s_2+d_2,d_2)}(\nu)$ are the same, while $\phi_{(s+d,d)}$ is injective for fixed $s$ and $d$.

In Remarks \ref{rem:oe}, \ref{rem:oo}, and \ref{rem:eo}, for given appropriate integers $s$, $d$, and $p$, we showed that there is a one-to-one correspondence between the set $\mathcal{SC}_{(s,s+d,\dots,s+pd)}$ and the set of functions $f$ with necessary conditions. Now, we are ready to prove our main result.


\begin{proof}[Proof of Theorem~\ref{thm:main}]
It follows almost directly from the construction of $\phi_{(s+d,d)}$ and Propositions \ref{prop:f_initial}, \ref{prop:f_oe}, \ref{prop:f_oo}, and \ref{prop:f_eo}. It is clear that $\phi_{(s+d,d)}(\la) \in \mathcal{F}(\lfloor s/2 \rfloor+\lceil d/2 \rceil, -\lceil d/2 \rceil)$ for $\la \in \mathcal{SC}_{(s,s+d,\dots,s+pd)}$. By Proposition~\ref{prop:f_initial}~(c), $p\geq 3$ and $f(j-1)=f(j)-1$ imply that $f(j-p+1),f(j-p+2),\dots,f(j-2)\geq f(j-1)$, for $p-1 \leq j\leq (s+d-1)/2$. It follows from the construction of $\phi_{(s+d,d)}$ that $\phi_{(s+d,d)}(\la)$ does not contain $UF^iU$ steps for all $i=0,1,\dots, p-3$ if $p\geq 3$. 

We now give a proof of (a). By Proposition~\ref{prop:f_oe}~(b), $f((s+d-1)/2-k-1)\geq -d/2$, for $k=0,1,\dots, \lfloor (p-3)/2 \rfloor$ if $p\geq 3$. It follows from the construction of $\phi_{(s+d,d)}$ that $\phi_{(s+d,d)}(\la)$ never ends with $UF^k$ steps for $k=0,1,\dots, \lfloor (p-3)/2 \rfloor$ if $p\geq 3$. By Proposition~\ref{prop:f_oe}~(c), $f(\ell+1)\leq 0$, for all $\ell=0,1,\dots,\lfloor (p-4)/2 \rfloor$ if $p\geq 4$. This implies that $\phi_{(s+d,d)}(\la)$ cannot start with $F^jU$ steps for all $j=0,1,\dots,\lfloor (p-4)/2 \rfloor$ if $p\geq 4$. Note that
(b) and (c) can be proved in a similar manner by using Propositions~\ref{prop:f_oo}~(b) and (c), and \ref{prop:f_eo}~(b) and (c), respectively.
\end{proof}

\begin{example}
Figure~\ref{fig:motzkin2} shows the corresponding free rational Motzkin paths associated with the $(s+d,d;a)$-abacus functions of  the self-conjugate partitions $\la$, $\mu$, and $\nu$ in the previous examples. Let $P=\phi_{(25,4)}(\la)$. Then $P=FDUFFUDDDDUF\in \mathcal{F}(12,-2)$ is a path satisfying that i) $P$ starts with no $U$, ii) $P$ ends with no $U$ or $UF$, iii) $P$ has no $UU$ or $UFU$ as a consecutive subpath. Similarly, one can check that $\phi_{(26,3)}(\mu)=\phi_{(25,3)}(\nu)=FDUFFUDDDDUFF$\\$\in \mathcal{F}(13,-2)$ satisfies conditions given in Theorem \ref{thm:main}.
\end{example}

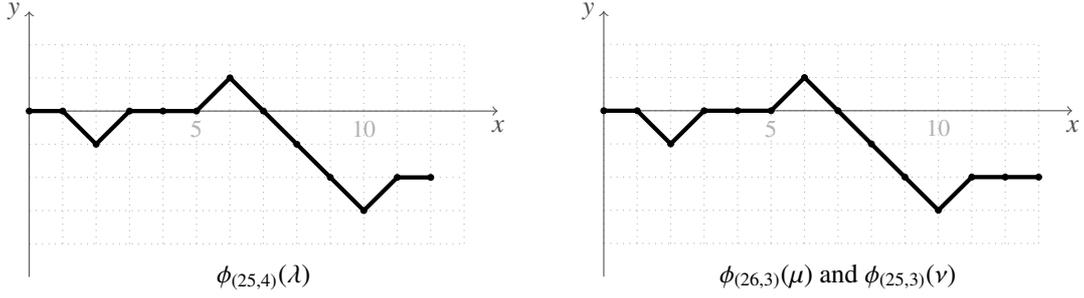
\begin{figure}[ht!]
\centering
\begin{tikzpicture}[scale=.44]
\node at (0,-4.8) {};
\node[below, gray!60] at (5,0) {\small$5$};
\node[below, gray!60] at (10,0) {\small$10$};
\node at (7,-5) {$\phi_{(25,4)}(\la)$};

\foreach \i in {0,1,2,...,13}
\draw[dotted, gray!60] (\i,-4)--(\i,2);

\foreach \i in {-4,-3,...,2}
\draw[dotted, gray!60] (0,\i)--(13,\i);

\draw[->, black!70] (0,0)--(14,0);
\draw[->, black!70] (0,-5)--(0,3);

\node[below, black!70] at (14,0) {$x$};
\node[left, black!70] at (0,3) {$y$};

\draw[ultra thick] (0,0)--(1,0)--(2,-1)--(3,0)--(4,0)--(5,0)--(6,1)--(7,0)--(8,-1)--(9,-2)--(10,-3)--(11,-2)--(12,-2);

\filldraw (0,0) circle (2.5pt);
\filldraw (1,0) circle (2.5pt);
\filldraw (2,-1) circle (2.5pt);
\filldraw (3,0) circle (2.5pt);
\filldraw (4,0) circle (2.5pt);
\filldraw (5,0) circle (2.5pt);
\filldraw (6,1) circle (2.5pt);
\filldraw (7,0) circle (2.5pt);
\filldraw (8,-1) circle (2.5pt);
\filldraw (9,-2) circle (2.5pt);
\filldraw (10,-3) circle (2.5pt);
\filldraw (11,-2) circle (2.5pt);
\filldraw (12,-2) circle (2.5pt);
\end{tikzpicture}
\qquad
\begin{tikzpicture}[scale=.44]
\node at (0,-4.8) {};
\node[below, gray!60] at (5,0) {\small$5$};
\node[below, gray!60] at (10,0) {\small$10$};
\node at (7,-5) {$\phi_{(26,3)}(\mu)$ and $\phi_{(25,3)}(\nu)$};

\foreach \i in {0,1,2,...,13}
\draw[dotted, gray!60] (\i,-4)--(\i,2);

\foreach \i in {-4,-3,...,2}
\draw[dotted, gray!60] (0,\i)--(13,\i);

\draw[->, black!70] (0,0)--(14,0);
\draw[->, black!70] (0,-5)--(0,3);

\node[below, black!70] at (14,0) {$x$};
\node[left, black!70] at (0,3) {$y$};

\draw[ultra thick] (0,0)--(1,0)--(2,-1)--(3,0)--(4,0)--(5,0)--(6,1)--(7,0)--(8,-1)--(9,-2)--(10,-3)--(11,-2)--(12,-2)--(13,-2);

\filldraw (0,0) circle (2.5pt);
\filldraw (1,0) circle (2.5pt);
\filldraw (2,-1) circle (2.5pt);
\filldraw (3,0) circle (2.5pt);
\filldraw (4,0) circle (2.5pt);
\filldraw (5,0) circle (2.5pt);
\filldraw (6,1) circle (2.5pt);
\filldraw (7,0) circle (2.5pt);
\filldraw (8,-1) circle (2.5pt);
\filldraw (9,-2) circle (2.5pt);
\filldraw (10,-3) circle (2.5pt);
\filldraw (11,-2) circle (2.5pt);
\filldraw (12,-2) circle (2.5pt);
\filldraw (13,-2) circle (2.5pt);
\end{tikzpicture}

\caption{The corresponding free rational Motzkin paths of $\la$, $\mu$, and $\nu$}\label{fig:motzkin2}
\end{figure}

From Theorem~\ref{thm:main}, we easily get the following result.

\begin{cor}\label{cor:d_odd}
Let $s$ and $d$ be relatively prime positive integers. If $d$ is odd and $s, p$ are even, then 
the number of self-conjugate $(s,s+d, \dots, s+pd)$-core partitions is equal to that of $(s+1,s+d+1, \dots, s+pd+1)$-core partitions.
\end{cor}


\subsection{Self-conjugate $(s,s+d,s+2d)$-core partitions and self-conjugate $(s,s+d,s+2d,s+3d)$-core partitions}\label{sec:3.2}
In particular, we give a closed formula for the number of self-conjugate $(s,s+d,s+2d)$-core partitions.

\begin{thm}\label{thm:count2}
Let $s$ and $d$ be relatively prime positive integers. The number of self-conjugate $(s,s+d,s+2d)$-core partitions is given by 
\[
\sum_{i=0}^{\lfloor \frac{s}{4} \rfloor} \binom{\frac{s+d-1}{2} }{i,\frac{d}{2}+i,  \frac{s-1}{2}-2i}\,, \quad \text{if $d$ is even;}
\]
\[
\sum_{i=0}^{\lfloor \frac{s}{2}\rfloor} \binom{\lfloor \frac{s+d-1}{2}\rfloor}{\lfloor \frac{i}{2} \rfloor, \lfloor \frac{d+i}{2}\rfloor, \lfloor \frac{s}{2} \rfloor -i}\,, \quad \text{if $d$ is odd.}
\]
\end{thm}

\begin{proof}
If $s=2r+1$ and $d=2c$, then the number of self-conjugate $(s,s+d,s+2d)$-core partitions is equal to the number of free rational Motzkin paths of type $(r+c,-c)$ by Theorem~\ref{thm:main}. 
For $0\leq i \leq \lfloor r/2 \rfloor$, the number of free rational Motzkin paths of type $(r+c,-c)$ having $i$ up steps (so that it has $c+i$ down steps and $r-2i$ flat steps) is given by $\binom{r+c}{i,c+i,r-2i}$. Hence, the number of 
free rational Motzkin paths of type $(r+c,-c)$ is given by 
\[
\sum_{i=0}^{\lfloor \frac{r}{2} \rfloor} \binom{r+c}{i,c+i,r-2i}
=\sum_{i=0}^{\lfloor \frac{s}{4} \rfloor} \binom{\frac{s+d-1}{2} }{i,\frac{d}{2}+i,  \frac{s-1}{2}-2i}\,.
\]

On the other hand, Theorem~\ref{thm:main} says that if $s=2r+1$ and $d=2c-1$, then the number of self-conjugate $(s,s+d,s+2d)$-core partitions is equal to the number of free rational Motzkin paths of type $(r+c,-c)$ for which ends with either a down step or a flat step.  
Since the number of free rational Motzkin paths of type $(r+c,-c)$ with $k$ up steps for which ends with a down (resp. flat) step is given by $\binom{r+c-1}{k,c+k-1,r-2k}$ (resp. $\binom{r+c-1}{k,c+k,r-2k-1}$), the total number of such paths is 
\[
\sum_{k=0}^{\lfloor \frac{r}{2} \rfloor} \binom{r+c-1}{k,c+(k-1),r-2k} +\sum_{k=0}^{\lfloor \frac{r-1}{2}\rfloor} \binom{r+c-1}{k,c+k,r-(2k+1)}
=\sum_{i=0}^{r} \binom{r+c-1}{\lfloor \frac{i}{2} \rfloor, c+ \lfloor \frac{i-1}{2}\rfloor, r-i}\,.
\]
It follows from Corollary \ref{cor:d_odd} that if $s=2r$ and $d=2c-1$, then the number of self-conjugate $(s,s+d,s+2d)$-core partitions is also given by 
\[
\sum_{i=0}^{r} \binom{r+c-1}{\lfloor \frac{i}{2} \rfloor, c+ \lfloor \frac{i-1}{2}\rfloor, r-i}\,.
\]
Hence, we conclude that for odd $d$, the number of self-conjugate $(s,s+d,s+2d)$-core partitions is given by
\[
\sum_{i=0}^{\lfloor \frac{s}{2}\rfloor} \binom{\lfloor \frac{s+d-1}{2}\rfloor}{\lfloor \frac{i}{2} \rfloor, \lfloor \frac{d+i}{2}\rfloor, \lfloor \frac{s}{2} \rfloor -i}\,.
\]
This completes the proof
\end{proof}

Indeed, the result by putting $d=1$ in Theorem \ref{thm:count2} agrees with Theorem \ref{thm:symMotzkin}. We also give a closed formula for the number of self-conjugate $(s,s+d,s+2d,s+3d)$-core partitions.

\begin{thm}\label{thm:count3}
Let $s$ and $d$ be relatively prime positive integers. The number of self-conjugate $(s,s+d,s+2d,s+3d)$-core partitions is given by 
\[
\sum_{i=0}^{\lfloor\frac{s}{4}\rfloor}\binom{\frac{s+d-1}{2}-i}{\frac{s-1}{2}-2i} \binom{\frac{s+d-1}{2}-i}{i}\,, \quad \text{if $d$ is even;}
\]
\[
\sum_{i=0}^{\lfloor \frac{s}{2} \rfloor}\binom{\lfloor \frac{s+d-1}{2}\rfloor-\lfloor \frac{i}{2} \rfloor}{\lfloor \frac{s}{2} \rfloor-i}\binom{\lfloor \frac{s+d}{2} \rfloor-\lfloor \frac{i+1}{2} \rfloor}{\lfloor \frac{i}{2} \rfloor}\,,
\quad \text{if $d$ is odd.}
\]
\end{thm}

\begin{proof}
First, we consider the case where $s=2r+1$ and $d=2c$. It follows from Theorem \ref{thm:main} that the number of self-conjugate $(s,s+d,s+2d,s+3d)$-core partitions is equal to the number of free rational Motzkin paths of type $(r+c,-c)$ for which ends with no $U$ and has no $UU$ as a consecutive subpath. Among these corresponding paths, we focus on the paths $P$ with $i$ up steps. A path $P$ can be obtained as follows:
For a given path $Q=Q_1\cdots Q_{r+c-i}$ consisting of $c+i$ down steps and $r-2i$ flat steps, insert $i$ up steps in $Q$ satisfying that i) there is at most one up step before $Q_1$; ii) there is at most one up step between $Q_{j}$ and $Q_{j+1}$ for $j=1,\dots,r+c-i-1$; iii) there is no up step after $Q_{r+c-i}$. Then we have a free rational Motzkin paths of type $(r+c,-c)$ for which ends with no $U$ and has no $UU$ as we desired. Note that  
there are $\binom{r+c-i}{r-2i}$ ways to choose $Q$ and there are $\binom{r+c-i}{i}$ ways to  insert $i$ up steps satisfying the conditions. Hence, the total number of such $P$'s is given by
\[
\sum_{i=0}^{\lfloor\frac{r}{2}\rfloor}\binom{r+c-i}{r-2i}\binom{r+c-i}{i}=\sum_{i=0}^{\lfloor\frac{s}{4}\rfloor} \binom{\frac{s+d-1}{2}-i}{\frac{s-1}{2}-2i}\binom{\frac{s+d-1}{2}-i}{i}\,.
\]

Now, we consider the case where $s=2r+1$ and $d=2c-1$. Theorem \ref{thm:main} gives that the number of self-conjugate $(s,s+d,s+2d,s+3d)$-core partitions is equal to the number of free rational Motzkin paths of type $(r+c,-c)$ for which ends with no $U$ or $UF$ and has no $UU$ as a consecutive subpath. Among these corresponding paths, we focus on the paths $P$ with $k$ up steps as well. In this case, a path $P$ can be obtained as follows:
For a given path $Q=Q_1\cdots Q_{r+c-k}$ consisting of $c+k$ down steps and $r-2k$ flat steps, insert $k$ up steps in $Q$ satisfying that i) there is at most one up step before $Q_1$; ii) there is at most one up step between $Q_{j}$ and $Q_{j+1}$ for $j=1,\dots,r+c-k-1$; iii) there is no up step after $Q_{r+c-k}$; in addition, iv) there is no up step between $Q_{r+c-k-1}$ and $Q_{r+c-k}$ if $Q$ ends with a flat step.
Note that there are $\binom{r+c-k-1}{r-2k}$ ways to choose $Q$ ending with a down step and there are $\binom{r+c-k}{k}$ ways to insert $k$ up steps satisfying the conditions. On the other hand, there are $\binom{r+c-k-1}{r-2k-1}$ ways to choose $Q$ ending with a flat step and there are $\binom{r+c-k-1}{k}$ ways to insert $k$ up steps satisfying the conditions.
Hence, the total number of such $P$'s is given by
\[
\sum_{k=0}^{\lfloor\frac{r}{2}\rfloor}\binom{r+c-k-1}{r-2k}\binom{r+c-k}{k}
+\sum_{k=0}^{\lfloor\frac{r-1}{2}\rfloor}\binom{r+c-k-1}{r-(2k+1)}\binom{~r+c-k-1~}{k}
=\sum_{i=0}^{r}\binom{r+c-\lfloor \frac{i}{2} \rfloor -1}{r-i}\binom{r+c-\lfloor \frac{i+1}{2} \rfloor}{\lfloor \frac{i}{2} \rfloor}\,.
\]

	When $s=2r$ and $d=2c-1$, the paths $P$ can be obtained in a similar way. In this case, the paths are not supposed to start with $U$. Except for this, all conditions are the same as for the case where $s=2r-1$ and $d=2c-1$.
Hence, the number of $(s,s+d,s+2d,s+3d)$-core partitions is given by
\[
\sum_{k=0}^{\lfloor\frac{r}{2}\rfloor}\binom{r+c-k-1}{r-2k}\binom{r+c-k-1}{k}
+\sum_{k=0}^{\lfloor\frac{r-1}{2}\rfloor}\binom{r+c-k-1}{r-(2k+1)}\binom{r+c-k-2}{k}
=\sum_{i=0}^{r}\binom{r+c-\lfloor \frac{i}{2} \rfloor -1}{r-i}\binom{r+c-\lfloor \frac{i+1}{2} \rfloor-1}{\lfloor \frac{i}{2} \rfloor}\,.
\]

We note that for odd $d$, the number of self-conjugate $(s,s+d,s+2d,s+3d)$-cores can be written as
\[
\sum_{i=0}^{\lfloor \frac{s}{2} \rfloor}\binom{\lfloor \frac{s+d-1}{2}\rfloor-\lfloor \frac{i}{2} \rfloor}{\lfloor \frac{s}{2} \rfloor-i}\binom{\lfloor \frac{s+d}{2} \rfloor-\lfloor \frac{i+1}{2} \rfloor}{\lfloor \frac{i}{2} \rfloor}\,.
\]
\end{proof}


\subsection{Self-conjugate $(s,s+1,\dots, s+p)$-core partitions with $m$ corners}\label{sec:3.3}

For a partition $\la$, the number of corners in the Young diagram of $\la$ is equal to the number of distinct parts in $\la$. Huang-Wang \cite{HW} proved that the number of $(s,s+1)$-core partitions with $m$ corners is equal to the Narayana number $N(s,m+1)=\frac{1}{s}\binom{s}{m+1}\binom{s}{m}$, and the number of $(s,s+1,s+2)$-core partitions with $m$ corners is equal to $\binom{s}{2m}C_m$, where $C_m$ is the $m$th Catalan number. In \cite{CHS2}, the authors extended this result.

\begin{cor}\cite[Corollary 3.1]{CHS2}
For positive integers $s$, $p\geq2$, and $1\leq m\leq \lfloor s/2 \rfloor$, the number of $(s,s+1,\dots,s+p)$-core partitions with $m$ corners is \[
\sum_{\ell=0}^{r}N(m,\ell+1)\binom{s-\ell(p-2)}{2m}\,,
\]
where $r=\min(m-1,\,\lfloor (s-2m)/(p-2) \rfloor)$.
\end{cor}
 
In this subsection, we focus on self-conjugate $(s,s+1,\dots,s+p)$-core partitions with $m$ corners. For simplicity, let $F(s,p):=\phi_{(s+1,1)}(\mathcal{SC}_{(s,s+1,\dots,s+p)})$. Recall that, for positive integers $s$ and $p\geq2$, the set $F(s,p)$ consists of paths $P\in \mathcal{F}(\lfloor s/2 \rfloor +1, -1)$ satisfying that i) $P$ has no $UF^iU$ as a consecutive subpath for all $i=0,1,\dots,p-3$ if $p\geq 3$, ii) $P$ starts with no $F^jU$ for all $j=0,1,\dots,\lfloor (p-3)/2\rfloor$ if $p\geq 3$ (resp. $j=0,1,\dots,\lfloor (p-4)/2\rfloor$ if $p\geq 4$); iii) $P$ ends with no $UF^k$ for all $k=0,1,\dots, p-2$ when $s$ is even (resp. odd).
As a corollary of Theorem~\ref{thm:main}, we have a one-to-one correspondence between the set  $\mathcal{SC}_{(s,s+1,\dots, s+p)}$ and the set $F(s,p)$. Now we refine this one-to-one correspondence according to the number of corners in a self-conjugate $(s,s+1,\dots, s+p)$-core partition. 
For $A\in\{F,D\}$, let $F_{A}(s,p)$ denote the set of paths in the set $F(s,p)$ for which ends with a step $A$.

\begin{lem}\label{lem:one}
For positive integers $s$ and $p\geq2$, the mapping $\phi_{(s+1,1)}$ gives a one-to-one correspondence between the set of self-conjugate $(s,s+1,\dots, s+p)$-core partitions with an even (resp. odd) number of corners and the set $F_D(s,p)$ (resp. $F_F(s,p)$).
\end{lem}

\begin{proof}
Let $\la$ be a self-conjugate $(s,s+1,\dots, s+p)$-core partition and let $P=\phi_{(s+1,1)}(\la)$. We first note that $\la$ has an odd number of corners if and only if $1\in MD(\la)$.

Let $r=\lfloor s/2 \rfloor$. Note that, for the $(s+1,1;-2r-1)$-abacus diagram, position $(0,r)$ is labeled by $-1$ and it is the only position in column $r$ where beads are allowed to be placed. 
Hence, for the $(s+1,1;-2r-1)$-abacus function $f$ of $\la$, the value $f(r)$ is either $-1$ if $1\in MD(\la)$ or $0$ if $1\notin MD(\la)$. It follows from the construction of $P$, $P$ ends with either a down step if $1\notin MD(\la)$ or a flat step if $1\in MD(\la)$. This completes the proof.
\end{proof}

\begin{lem}\label{lem:flat}
Let $s$, $p$, and $D$ be positive integers with $D\geq 2$. For a self-conjugate $(s,s+1,\dots, s+p)$-core partition $\la$ with $MD(\la)=\{d_1,d_2, \dots, d_D\}$, where $d_1>d_2>\dots>d_D\geq1$, let $\tilde{\la}$ denote the self-conjugate partition with $MD(\tilde{\la})=\{d_2, \dots, d_D\}$. 
If $P=\phi_{(s+1,1)}(\la)$ and $\tilde{P}=\phi_{(s+1,1)}(\tilde{\la})$, then we have the following.

\begin{enumerate}
\item[(a)] If $d_1=d_2+2$, then $P$ and $\tilde{P}$ have the same number of flat steps.
\item[(b)] If $d_1>d_2+2$, then $P$ has two less flat steps than $\tilde{P}$.
\end{enumerate}

\end{lem}

\begin{proof}
If $s$ is even, it can be proven in a similar way to odd, so we only prove the odd case here.
Let $s=2r+1$. For the $(s+1,1;-s-1)$-abacus diagram, let position $(i,j)$ be labeled by either $d_1$ or $-d_1$ and let $f$ and $\tilde{f}$ be the $(s+1,1;-s-1)$-abacus function  of $\la$ and $\tilde{\la}$, respectively. We note that $1\leq j \leq r-1$ and $f(x)=\tilde{f}(x)$ for all $0\leq x \leq r+1$ except $x=j$ in any case. 

We first suppose that position $(i,j)$ is labeled by $d_1$ so that position $(i,j-1)$ is labeled by $d_1-2$.
\begin{itemize}
\item When $d_1=d_2+2$, $f(j-1)=f(j)=i,~ f(j+1)=i-1$ and $\tilde{f}(j-1)=i,~ \tilde{f}(j)=\tilde{f}(j+1)=i-1$. Hence, the paths can be written as $P=P_1 \cdots P_{j-1} FD P_{j+2} \cdots P_{r+1}$ and $\tilde{P}=P_1 \cdots P_{j-1} DF P_{j+2} \cdots P_{r+1}$. 
\item When $d_1>d_2+2$, $f(j-1)=i-1,~ f(j)=i,~ f(j+1)=i-1$ and $\tilde{f}(j-1)= \tilde{f}(j)=\tilde{f}(j+1)=i-1$. In this case, the paths can be written as $P=P_1 \cdots P_{j-1} UD P_{j+2} \cdots P_{r+1}$ and $\tilde{P}=P_1 \cdots P_{j-1} FF P_{j+2} \cdots P_{r+1}$.
\end{itemize}

We now suppose that position $(i,j)$ is labeled by $-d_1$ so that position $(i,j+1)$ is labeled by $-(d_1-2)$. 
\begin{itemize}
\item When $d_1=d_2+2$, $f(j-1)=i,~ f(j)=f(j+1)=i-1$ and $\tilde{f}(j-1)=\tilde{f}(j)=i,~\tilde{f}(j+1)=i-1$. Hence, the paths can be written as $P=P_1 \cdots P_{j-1} DF P_{j+2} \cdots P_{r+1}$ and $\tilde{P}=P_1 \cdots P_{j-1} FD P_{j+2} \cdots P_{r+1}$. 
\item When $d_1>d_2+2$, $f(j-1)=i,~ f(j)=i-1,~ f(j+1)=i$ and $\tilde{f}(j-1)= \tilde{f}(j)=\tilde{f}(j+1)=i$. In this case, the paths can be written as $P=P_1 \cdots P_{j-1} DU P_{j+2} \cdots P_{r+1}$ and $\tilde{P}=P_1 \cdots P_{j-1} FF P_{j+2} \cdots P_{r+1}$.
\end{itemize}
This completes the proof.
\end{proof}

\begin{thm}\label{thm:corner} 
For positive integers $s$ and $p\geq 2$, the mapping $\phi_{(s+1,1)}$ gives a one-to-one correspondence between the set of self-conjugate $(s,s+1,\dots,s+p)$-core partitions with an even (resp. odd) number $m$ of corners and the set of paths in $F_D(s,p)$ (resp. $F_F(s,p)$) having $\lfloor s/2 \rfloor-m$ (resp. $\lfloor s/2 \rfloor -m+1$) flat steps.
\end{thm}

\begin{proof}
Let $r=\lfloor s/2 \rfloor$. For a self-conjugate partition $\la$ with $MD(\la)=\{d_1,\dots,d_D\}$, where $d_1> \cdots >d_D$, we denote the self-conjugate partition with $MD(\tilde{\la})=\{d_2,\dots,d_D\}$ by $\tilde{\la}$.  

From Lemma~\ref{lem:one}, we have learned that there is a one-to-one correspondence between the set of self-conjugate $(s,s+1,\dots,s+p)$-core partitions with an even (resp. odd) number of corners and the set $F_D(s,p)$ (resp. $F_F(s,p)$).
To prove the remainder of the theorem, we claim that if $\la$ is a self-conjugate $(s,s+1,\dots,s+p)$-core partition with an even (resp. odd) number of corners, say $m$, then its corresponding path $P=\phi_{(s+1,1)}(\la)$ has $r-m$ (resp. $r-m+1$) flat steps. 
We prove the claim by using induction on $|MD(\la)|$. Let $f$ be the $(s+1,1;-2r-1)$-abacus function of $\la$.

If $|MD(\la)|=0$, then $\la$ is an empty partition so that it has $0$ corner. Since there is no bead on the $(s+1,1;-2r-1)$-abacus of $\la$, $f(j)=0$ for all $j$ except $f(r+1)=-1$. Hence, the path $P=F\cdots FD$ has $r-0$ flat steps as we claimed. 
We now consider the case with $|MD(\la)|=1$. Let $\la$ be a self-conjugate $(s,s+1,\dots,s+p)$-core partition with $MD(\la)=\{d\}$. We note that $d$ is less than $2s$, because that $d-2s\in MD(\la)$ otherwise. Hence, if position $(i,j)$ is labeled by either $d$ or $-d$, then $i$ must be either $0$ or $1$. First, we suppose that $d=1$ so that $j=r$ and $f(0)=\cdots=f(r-1)=0$ and $f(r)=f(r+1)=-1$. In this case, $\la=(1)$ has one corner and the corresponding path $P=F\cdots FDF$ has $r-1+1$ flat steps as we desired. If $d\neq 1$, then $\la$ has two corners and its corresponding path can be written as either $P=F\cdots F UD F\cdots F D$ or $P=F\cdots F DU F\cdots FD$ so that it has $r-2$ flat steps as desired. 

Now, we assume that the claim holds for all $\la$ with $|MD(\la)|=D-1$ for $D\geq2$. Let $\la$ be a self-conjugate $(s,s+1,\dots, s+p)$-core partition with $m$ corners and $MD(\la)=\{d_1,\dots,d_D\}$, where $d_1> \cdots >d_D$. When $d_1=d_2+2$, $\tilde{\la}$ has $m$ corners and its corresponding path $\tilde{P}$ has $r-m$ (resp. $r-m+1$) flat steps if $m$ is even (resp. odd) by the induction hypothesis. It follows from Lemma~\ref{lem:flat}~(a) that $P$ also has $r-m$ (resp. $r-m+1$) flat steps if $m$ is even (resp. odd). When $d_1>d_2+2$, $\tilde{\la}$ has $m-2$ corners and its corresponding path $\tilde{P}$ has $r-m+2$ (resp. $r-m+3$) flat steps if $m$ is even (resp. odd) by the induction hypothesis. It follows from Lemma~\ref{lem:flat}~(b) that $P$ has $r-m$ (resp. $r-m+1$) flat steps if $m$ is even (resp. odd). This completes the proof of the claim.

\end{proof}

In particular, we obtain formulae for the number of self-conjugate $(s,s+1,\dots,s+p)$-core partitions with $m$ corners for $p=2$ and $3$.

\begin{prop}\label{prop:corner2}
The number of self-conjugate $(s,s+1,s+2)$-core partitions with $m$ corners is given by 
\[
\binom{\lfloor \frac{s}{2} \rfloor}{\lfloor \frac{m}{2} \rfloor , \lfloor \frac{m+1}{2} \rfloor , \lfloor \frac{s}{2} \rfloor-m}\,.
\] 
\end{prop}

\begin{proof}
Let $r=\lfloor s/2 \rfloor$.
By Theorem~\ref{thm:corner}, if $m$ is even (resp. odd), then the number of self-conjugate $(s,s+1,s+2)$-core partitions with $m$ corners is equal to the number of free rational Motzkin paths of type $(r+1,-1)$ with $r-m$ (resp. $r-m+1$) flat steps for which ends with a down (resp. flat) step. Now, we enumerate these Motzkin paths. 

When $m=2k$, each of the paths consists of $k$ up steps, $k+1$ down step, and $r-2k$ flat steps and ends with a down step. Therefore, the number of such paths is $\binom{r}{k,k,r-2k}$. 

When $m=2k+1$, each of the paths consists of $k$ up steps, $k+1$ down step, and $r-2k$ flat steps and ends with a flat step. Therefore, the number of such paths is $\binom{r}{k,k+1,r-2k-1}$. 

Thus, the number of self-conjugate $(s,s+1,s+2)$-core partitions with $m$ corners is given by 
\[
\binom{r}{\lfloor \frac{m}{2} \rfloor , \lfloor \frac{m+1}{2} \rfloor , r-m}
\] 
and this completes the proof.
\end{proof}

\begin{prop}
The number of self-conjugate $(s,s+1,s+2,s+3)$-core partitions with $m$ corners is given by
\[
\binom{\lfloor \frac{s}{2}\rfloor-\lfloor \frac{m}{2} \rfloor}{\lfloor \frac{s}{2} \rfloor-m}\binom{\lfloor \frac{s+1}{2} \rfloor-\lfloor \frac{m+1}{2} \rfloor}{\lfloor \frac{m}{2} \rfloor}\,.
\]
\end{prop}

\begin{proof}
First, we consider the case where $s=2r+1$ and $m=2k$. It follows from Theorems \ref{thm:main} and \ref{thm:corner} that the number of self-conjugate $(s,s+1,s+2,s+3)$-core partitions with $m$ corners is equal to the number of free rational Motzkin paths of type $(r+1,-1)$ with $r-2k$ flat steps for which ends with a down step and has no $UU$ as a consecutive subpath. It can be found in the proof of Theorem \ref{thm:count3} that the number of such paths is $\binom{r-k}{r-2k}\binom{r+1-k}{k}$.

We now consider the case where $s=2r+1$ and $m=2k+1$. By Theorems \ref{thm:main} and \ref{thm:corner}, the number of self-conjugate $(s,s+1,s+2,s+3)$-core partitions with $m$ corners is equal to the number of free rational Motzkin paths of type $(r+1,-1)$ with $r-2k$ flat steps for which ends with a flat step, does not end with $UF$, and has no $UU$ as a consecutive subpath. Similarly, the number of such paths is $\binom{r-k}{r-(2k+1)}\binom{r-k}{k}$ by the proof of Theorem \ref{thm:count3}.

Therefore, the number of self-conjugate $(s,s+1,s+2,s+3)$-core partitions with $m$ corners is given by
\[
\binom{r-\lfloor \frac{m}{2} \rfloor}{r-m}\binom{r+1-\lfloor \frac{m+1}{2} \rfloor}{\lfloor \frac{m}{2} \rfloor}\,,
\]
where $s=2r+1$.

When $s=2r$, it is similar to the case where $s=2r+1$ that the number of self-conjugate $(s,s+1,s+2,s+3)$-core partitions is given by
\[
\binom{r-\lfloor \frac{m}{2} \rfloor}{r-m}\binom{r-\lfloor \frac{m+1}{2} \rfloor}{\lfloor \frac{m}{2} \rfloor}\,.
\]

Thus, the number of self-conjugate $(s,s+1,s+2,s+3)$-core partitions can be represented as
\[
\binom{\lfloor \frac{s}{2}\rfloor-\lfloor \frac{m}{2} \rfloor}{\lfloor \frac{s}{2} \rfloor-m}\binom{\lfloor \frac{s+1}{2} \rfloor-\lfloor \frac{m+1}{2} \rfloor}{\lfloor \frac{m}{2} \rfloor}\,.
\]

\end{proof}


\section*{Acknowledgments}
Hyunsoo Cho was supported by Basic Science Research Program through the National Research Foundation of Korea(NRF) funded by the Ministry of Education (Grant No. 2019R1A6A1A11051177). JiSun Huh was supported by the National Research Foundation of Korea(NRF) grant funded by the Korea government(MSIT) (No. 2020R1C1C1A01008524).

\bibliographystyle{plain}  
\bibliography{mybib} 

\end{document}